\documentclass[a4,11pt]{nmd/article}
\usepackage{graphicx,color}
\usepackage{epsfig,manfnt}
\usepackage{amsfonts}
\usepackage{amssymb}
\usepackage{amsmath}
\usepackage{amsthm}
\usepackage{latexsym}
\usepackage{amscd}
\usepackage{flafter}
\usepackage{epsf}
\usepackage{epstopdf}
\usepackage{inputenc}
\usepackage{stmaryrd}
\usepackage{tikz,tikz-cd}
\usepackage{mathrsfs}
\usepackage{lipsum}
\newcommand{\wPsi}{\widetilde{\Psi}}
\newcommand{\pz}{pr}
\newcommand{\Diff}{\mathrm{Diff}}
\newcommand{\tGamma}{\Lambda}

\newcommand{\vGamma}{\overline{\Gamma}}

\newcommand{\length}{\ell}

\newcommand{\jmap}{\jmath}

\newcommand{\vg}{{\overline{Y}}}

\newcommand{\x}{\mathrm{\bf{x}}}

\newcommand{\ra}{\rightarrow}

\newcommand{\Acal}{\mathcal{A}}

\newcommand{\Ecal}{\mathcal{E}}
\newcommand{\Fcal}{\mathcal{F}}
\newcommand{\Gcal}{\mathcal{G}}
\newcommand{\Hcal}{\mathcal{H}}

\newcommand{\Lcal}{\mathcal{L}}

\newcommand{\Pcal}{\mathcal{P}}

\newcommand{\Rcal}{\mathcal{R}}

\newcommand{\Ucal}{\mathcal{U}}
\newcommand{\Vcal}{\mathcal{V}}
\newcommand{\Wcal}{\mathcal{W}}
\newcommand{\Xcal}{\mathcal{X}}
\newcommand{\Ycal}{\mathcal{Y}}
\newcommand{\Zcal}{\mathcal{Z}}

\newcommand{\Ker}{\mathrm{Ker}}

\newcommand{\Coker}{\mathrm{Coker}}

\newtheorem{thm}{Theorem}[section]
\newtheorem{prop}[thm]{Proposition}
\newtheorem{cor}[thm]{Corollary}

\theoremstyle{definition}

\newtheorem{defn}[thm]{Definition}

\begin{document}
\title[Counting closed geodesics on Riemannian manifolds]
{Counting closed geodesics on Riemannian manifolds }%
\author{Eaman Eftekhary}%
\address{School of Mathematics, Institute for Research in 
Fundamental Sciences (IPM), P. O. Box 19395-5746, Tehran, Iran}%
\email{eaman@ipm.ir}
\begin{abstract}
Fix a smooth closed manifold $M$. Let $\Rcal_M$ denote the space of all pairs $(g,L)$ such that $g$ is a $C^3$ Riemannian metric on $M$ and $L\in\R$ is not the length of any closed $g$-geodesics. A locally constant {\emph{geodesic count function}} $\pi_M:\Rcal_M\ra \Z$ is constructed. For this purpose, the {\emph{weight}} of   compact open subsets of the space of  closed $g$-geodesics is defined and investigated for an arbitrary Riemannian metric $g$. 
\end{abstract}
\maketitle
\section{Introduction}
Closed geodesics on negatively curved Riemannian manifolds have been investigated for decades. Philips and Sarnak obtained the asymptotics of the count function over manifolds with constant negative curvature \cite{PS}. Their correspondence between the count function and the topological entropy of the metric was  generalized to manifolds with variable negative curvature by Margulis \cite{Margulis}. The result was sharpened by Anantharaman \cite{Anan} and by Pollicott and Sharp \cite{PS-1}.\\

 There are serious obstacles when the same problem is studied over Riemannian manifolds with arbitrary curvature. In particular,  the space of closed geodesics is not discrete for an arbitrary metric and may include infinite families. Although such infinite families correspond to discrete sets of closed geodesics when the metric is perturbed to a {\emph{generic}} metric in its vicinity, different perturbations give structurally different discrete sets of closed geodesics. The goal of this paper is a careful comparison between such generic perturbations, which leads to the definition of a  count function satisfying certain natural properties (see Theorem~\ref{thm:main-intro} and Corollary~\ref{cor:main-intro} below). \\

Fix a smooth closed manifold $M$ of dimension $n$  throughout the paper. The space of $C^3$ Riemannian metrics on $M$ is denoted by $\Gcal$ and the space of $W^{2,2}$ maps from $S^1$ to $M$ is denoted by $\Xcal$. Let $\Zcal\subset\Gcal\times \Xcal$ consist of all pairs $(g,\gamma)$ such that $\gamma$ is a $g$-geodesic, and  $\Ycal\subset\Zcal$ consist of $(g,\gamma)$ such that $\gamma$ is not multiply covered. Abraham showed that $\Ycal$ is a separable $C^1$ Banach manifold \cite{Abraham}. If $\pz:\Zcal\ra\Gcal$ and $p=\pz|_{\Ycal}:\Ycal\ra \Gcal$ are the projection maps,  let $\Zcal(g)=\pz^{-1}(g)$ denote the space of all closed $g$-geodesics and $\Ycal(g)=p^{-1}(g)\subset \Zcal(g)$ be the ones which are not multiply covered. $\Ycal(g)$ and $\Gcal(g)$ are usually identified as subsets of $\Xcal$. Abraham also proved that for a generic (or {\emph{bumpy}}) metric $g$ in $\Gcal$, there are no non-trivial Jacobi vector fields along closed $g$-geodesics \cite{Abraham} (see also \cite{White-1}). Jacobi fields along $\gamma\in\Zcal(g)$ describe the closed geodesics which are {infinitesimally close} to $\gamma$. In particular, trivial Jacobi fields along $\gamma$ (i.e. constant multiples of $\dot\gamma$) correspond to the rotation action of $S^1$ on $\Zcal$. The orbit of $(g,\gamma)\in\Zcal$ under this action is denoted by $(g,[\gamma])$. Abraham's result implies that for a generic (bumpy) metric $g$, the quotient $\widetilde{\Zcal}(g)=\Zcal(g)/S^1$ is discrete.\\

However, $\widetilde{\Zcal}(g)$ is not discrete for  arbitrary  $g\in\Gcal$, and may include infinite families. To define the {\emph{geodesic count function}} $\pi_M$,  the {\emph{weight}} $n(g,\Gamma)$ of a compact and open subset $\Gamma$ of $\Zcal(g)$ should be defined. If $g'$ is a bumpy metric which is sufficiently close to $g$, $\Gamma$ determines a compact and open subset $\Gamma'\subset\Zcal(g')$ such that $\Gamma'/S^1$ is finite. The number of geodesics in $\Gamma'/S^1$ (counted with appropriate signs) is one first candidate for  $n(g,\Gamma)$. Nevertheless, even the parity of the number of points in $\Gamma'/S^1$ depends on the choice of $g'$ near $g$,  based on a Kuranishi model developed in Section~\ref{sec:invariance}. This may be linked to the fact that $\pz:\Zcal\ra\Gcal$ is not locally proper. \\

Besides defining  $n(g,\Gamma)$ in the above context, one also needs a notion of invariance when $g$ moves in $\Gcal$. Given $g_0,g_1\in\Gcal$, let $\Pcal_{g_0,g_1}$ denote the space of paths $\vg:[0,1]\ra \Gcal$ with $\vg(0)=g_0$ and $\vg(1)=1$, which are of class $C^3$ as $2$-tensors on $[0,1]\times M$. Let $\Zcal_{g_0,g_1}$ and $\Ycal_{g_0,g_1}$ be the set of all triples  
\[(\vg,t,\gamma)\in\Pcal_{g_0,g_1}\times[0,1]\times\Xcal\] 
 such that $(\vg(t),\gamma)$ is in $\Zcal$ and $\Ycal$, respectively.  The fibers of $\Zcal_{g_0,g_1}$ and $\Ycal_{g_0,g_1}$ over  $\vg\in\Pcal_{g_0,g_1}$ are denoted by  
\[\Zcal(\vg),\Ycal(\vg)\subset [0,1]\times \Xcal,\] 
respectively. Informally, $\Zcal(\vg)$  is a {\emph{cobordism}} connecting  $\Zcal(g_0)$ to $\Zcal(g_1)$  and  $\Ycal(\vg)$ is a {\emph{cobordism}} connecting  $\Ycal(g_0)$ to $\Ycal(g_1)$.
\begin{thm}\label{thm:main-intro}
For every Riemnnian metric $g\in\Gcal$ and every open and compact subset $\Gamma$ of $\Zcal(g)$, the weight $n(g,\Gamma)\in\Z$ of $\Gamma$  may be defined so that:
\begin{itemize}
	\item If $g$ is a bumpy metric and $\gamma\in\Zcal(g)$ is not multiply covered,  $n(g,[\gamma])=\pm 1$. Further, if $g$ is negatively curved near $\gamma$,  $n(g,[\gamma])=1$ and for every multiple cover $\gamma'$ of $\gamma$, $n(g,[\gamma'])=0$.
	\item If $\vg$ is a path in $\Pcal_{g_0,g_1}$ and  $\vGamma\subset\Zcal(\vg)$ is  open and compact,  then for $j=0,1$ the subset $\Gamma_j=\vGamma\cap\left(\{j\}\times\Xcal\right)$ of $\Zcal(g_j)$ is open and compact, and we have $n(g_0,\Gamma_0)=n(g_1,\Gamma_1)$.
\end{itemize}
\end{thm}	

Let $\length_g:\Zcal(g)\ra \R$ measure the lengths of closed geodesics and $\Lcal_g$ denote the closed image of $\length_g$, which has zero Lebesgue measure \cite[Theorem~16.2]{Milnor-MT}.  Let $\Rcal_M\subset\Gcal\times\R$ consist of all pairs $(g,L)$ with $L\notin\Lcal_g$. $\Gamma_{g,L}=\length_g^{-1}(0,L)$ is a compact and open  subset of $\Zcal(g)$ for $(g,L)\in\Rcal_M$. The {\emph{geodesic count function}} $\pi_M:\Rcal_M\ra\Z$ is defined by  $\pi_M(g,L)=n(g,\Gamma_{g,L})$.  Theorem~\ref{thm:main-intro} implies:

\begin{cor}\label{cor:main-intro} 
The {{geodesic count function}} $\pi_M:\Rcal_M\ra \Z$  is  locally constant on $\Rcal_M$.
\end{cor}

For every $g\in\Gcal$, the function $\pi_g=\pi_M(g,\cdot):\R\setminus\Lcal_g\ra \Z$ is well-defined. Corollary~\ref{cor:main-intro} and the study of the geodesic count function over negatively curved Riemannian manifolds suggest the study of the asymptotic behavior of $\pi_g$ and comparing it with the topological entropy of $g$. \\

Another interesting case of our construction is the situation where  $U\subset \Gcal$ is open and the restriction  $p_U:p^{-1}(U)\ra U$ of $p$ is proper. In this case,  $\Ycal(g)\subset \Xcal$ is open an compact for every $g\in U$ and the weight $n(g,\Ycal(g))$ is well-defined. If $U$ is also connected, this number is constant and is precisely the {\emph{degree}} of $p_U$, as defined and investigated in \cite{White-1}. One example is the case where $U$ is a neighborhood of the standard (round) metric on the sphere $M=S^n$.\\

 Given $d\in\Z^+$, let $\varphi_d:S^1\ra S^1$ be the standard degree $d$ map, $\psi_d:S^1\ra S^1$ be the rotation by $2\pi/d$, and $\lambda$ be a primitive $d$-th root of unity. Given $(g,\gamma)\in\Ycal$, a non-zero normal Jacobi field $\xi$ along $\gamma\circ\varphi_d:S^1\ra M$ is called a $\lambda$-Jacobi field along $\gamma$, if it admits a {\emph{twin}} normal Jacobi field $\xi'$ such that $\psi_d^*(\xi+i\xi')=\lambda\cdot (\xi+i\xi')$ (where $i=\sqrt{-1}$ throughout the paper). The existence of a normal Jacobi field along a multiple cover of $(g,\gamma)\in\Ycal$ implies the existence of a $\lambda$-Jacobi field along $(g,\gamma)$ for some root $\lambda$ of unity. The {\emph{bad locus}}, where closed geodesics are accompanied with $\lambda$-Jacobi fields (for a fixed root $\lambda$ of unity) is described by the following proposition:

\begin{prop}\label{prop:main-intro}
 The triples $(g,\gamma,\xi)$ where $(g,\gamma)\in\Ycal$ and $\xi$ is a  $\lambda$-Jacobi field  along $\gamma$ form a $C^{1}$ Banach manifold $\Wcal^\lambda$, and is equipped with an action of $S^1\times\R^*$ (which rotates the geodesics and scales the Jacobi fields). The fibering map  
 \[p^\lambda:\widetilde{\Wcal}^\lambda=\Wcal^\lambda/(S^1\times\R^*)\ra\Gcal\] 
  is Fredholm and has index $-1$. 
\end{prop}

Proposition~\ref{prop:main-intro} implies that the image of $\widetilde{\Wcal}^\lambda$ in $\Gcal$ has codimension $2$ for $\lambda\neq \pm1$, and can be avoided when  two bumpy metrics $g_0$ and $g_1$ are connected with a generic path $\vg$. However, $\pm1$-Jacobi fields are inevitable along a discrete subset $\widetilde{\Zcal}_J(\vg)\subset\widetilde{\Zcal}(\vg)$. A Kuranishi model for a neighborhood of $\widetilde{\Zcal}_J(\vg)$ in $\widetilde{\Zcal}(\vg)$ is developed,  and is key to the proof of Theorem~\ref{thm:main-intro}.  The model identifies the aforementioned neighborhood with the zero set of a function $F:\R^2\ra \R$ and shows that $\widetilde{\Zcal}(\vg)$ is not a manifold near $(t,\gamma\circ\varphi_2)$ if there is a $-1$-Jacobi field $\xi$  along $(t,\gamma)\in\Ycal(\vg)$. However, the weights are defined so that the invariance claim in Theorem~\ref{thm:main-intro} is achieved. \\

Let $Z$ denote the zero section in the tangent bundle $TM$ of the $n$-manifold $M$ and $\sim$ denote the equivalence relation which identifies any vector in $TM\setminus Z$ with its non-zero multiples. The quotient $S_M=(TM\setminus Z)/\sim$ is then a smooth closed  manifold of dimension $2n-1$, which may be identified as the unit tangent  bundle of $M$ in the presence of a Riemannian metric. Denote the space of smooth vector fields on $S_M$ by $\Vcal(S_M)$. Every Riemannian metric $g\in\Gcal(M)$ corresponds to a vector field $\zeta_g\in\Vcal(S_M)$ and closed geodesics on the Riemannian manifold $(M,g)$ are  in correspondence with closed orbits of the geodesic flow (on $S_M$) associated with the Riemannian metric $g$ (and the vector field $\zeta_g$).  Counting closed geodesics on $M$ is thus closely related to counting periodic orbits of vector fields on $S_M$, which is discussed in \cite{Ef-orbits}. We obtain a similar count function, which is related to some of the standard fixed point theorems for dynamical systems. The possibility of having zeros for general vector fields makes the discussion of \cite{Ef-orbits} more complicated.

\section{Preliminaries on the moduli space of closed geodesics}\label{sec:moduli}
We continue to use the notation set in the introduction. For $\gamma\in\Xcal$, the equivalence class of $\gamma$ under reparametrization of the domain $S^1$ is denoted by $[\gamma]\in\Xcal/\mathrm{Diff}^+(S^1)$. The pull-back of the tangent bundle $TM$ of $M$  under  $\gamma$ is a vector bundle  $W_\gamma=\gamma^*TM$ over $S^1$. The image of the tangent bundle  of $S^1$ under $d\gamma$ is a line sub-bundle $T_\gamma$ of $W_\gamma$ and  $N_\gamma=W_\gamma/T_\gamma$ is the normal bundle of $\gamma$ in  $M$.  Given $k\in\Z^{\geq 0}$, let $A_{k}(W_\gamma)$ and $A_{k}(N_\gamma)$ denote the vector spaces of  $W_\gamma$-valued and $N_\gamma$-valued functions on $S^1$, respectively, which are  of Sobolev class $W^{k,2}$.  $A_{2}(W_\gamma)$ is the tangent fiber $T_\gamma\Xcal$.  For $\gamma\in\Xcal$, let $\dot\gamma=d\gamma/d\theta$, where $\theta$ parametrizes the circle $S^1=\R/\Z$.  Given $g\in\Gcal$, a map $\gamma\in\Xcal$  is a {\emph{closed $g$-geodesic}} if  $\nabla^g_{\dot{\gamma}}\dot{\gamma}=0$, where  $\nabla^g$ denotes the Levi-Civita connection of $g$. If $\gamma\circ\rho\in[\gamma]$ is also a $g$-geodesic for some $\rho\in\Diff^+(S^1)$, it follows that $\rho$ is a rotation.  The equivalence class of $(g,\gamma)\in\Zcal$ under the rotation action of $S^1$ may thus be denoted $(g,[\gamma])$. The closed $g$-geodesics which share the same image in $M$ form a union of two copies of $S^1$ in $\Zcal$. If $R^g$ is the Riemann curvature of $g$, given by 
\[R^g(X,Y)Z=\nabla^g_X\nabla^g_Y Z-\nabla^g_Y\nabla^g_X Z-\nabla^g_{[X,Y]}Z,\] 
we define 
\begin{displaymath}
\Psi^w_{g,\gamma}:A_{2}(W_\gamma)\ra A_{0}(W_\gamma),\quad\quad\Psi^w_{g,\gamma}(\zeta)=\nabla^g_{\dot{\gamma}}\nabla^g_{\dot{\gamma}}\zeta+R^g(\zeta,\dot{\gamma})\dot{\gamma}.
\end{displaymath}
The operator $\Psi^w_{g,\gamma}$ induces the maps 
\[\Psi^t_{g,\gamma}:A_{2}(T_\gamma)\ra A_{0}(T_\gamma)\quad\text{and}\quad\Psi_{g,\gamma}:A_{2}(N_\gamma)\ra A_{0}(N_\gamma).\] 
By a {\emph{Jacobi field}} $\zeta$ along $\gamma$,  we mean a vector field in the kernel of  $\Psi^w_{g,\gamma}$. The Jacobi fields describe the difference between $\gamma$ and an infinitesimally close geodesic. Constant multiples of $\dot{\gamma}$ are {\emph{trivial}} examples of Jacobi fields, and correspond to the rotation action of $S^1$ on $\Zcal$. If there are no non-trivial Jacobi fields along $\gamma$, then $\gamma$ is isolated in $\widetilde{\Zcal}(g)=\Zcal(g)/S^1$. This happens if and only if the kernel of $\Psi_{g,\gamma}$ is trivial. If this latter condition is satisfied for every closed $g$-geodesics, $g\in\Gcal$ is called a {\emph{bumpy metric}}.  Theorem~\ref{thm:transversality} follows from the argument of \cite{Abraham} (c.f. \cite{White-1} and \cite{White-2}).

 \begin{thm}\label{thm:transversality}
$\Ycal$ is a separable $C^{1}$ Banach manifold. For $g$ in the Bair subset $\Gcal^*\subset \Gcal$  of regular values of the Fredholm map $p:\Ycal\ra\Gcal$,  $\widetilde{\Ycal}(g)=\Ycal(g)/S^1$ is  a $0$-manifold.  Moreover, for $g_0,g_1\in\Gcal^{*}$,  $\Ycal_{g_0,g_1}$ is a separable $C^{1}$ Banach manifold. For $\vg$ in the Bair subset $\Pcal^{*}_{g_0,g_1}\subset \Pcal_{g_0,g_1}$ of regular values of the Fredholm fibering map from $\Ycal_{g_0,g_1}$ to $\Pcal_{g_0,g_1}$,  $\widetilde{\Ycal}(\vg)=\Ycal(\vg)/S^1$ is a $C^1$ $1$-manifold.
\end{thm}

There is a bundle $\Ecal\ra \Gcal\times\Xcal$ with fiber $\Ecal_{g,\gamma}=A_0(W_\gamma)$ at $(g,\gamma)\in\Gcal\times\Xcal$, and $\Zcal$ is the zero locus of the section $\nabla$ of $\Ecal$ defined by  $\nabla(g,\gamma)=\nabla^g_{\dot\gamma}\dot\gamma$. At a zero $(g,\gamma)$ of $\nabla$, 
the differential of $\nabla$ may be projected over the fiber to give a linear map
\[d\nabla_{g,\gamma}:T_g\Gcal\oplus T_\gamma \Xcal=\Hcal\oplus A_2(W_\gamma)\ra \Ecal_{g,\gamma}=A_0(W_\gamma).\]
Here $\Hcal=T_g\Gcal$ is the space of symmetric $2$-tensors of class $C^{2}$ on $M$. $d\nabla$ induces a map 
\[\wPsi_{g,\gamma}:\Hcal\oplus A_2(N_\gamma)\ra A_0(N_\gamma),\quad\quad\wPsi(h,\xi)=\Psi^m_{g,\gamma}(h)+\Psi_{g,\gamma}(\xi),\]
where $\Psi^m_{g,\gamma}(h)$ is the image of $d\nabla_{g,\gamma}(h)$ in $A_0(N_\gamma)$. If $\wPsi_{g,\gamma}$ is surjective, then $\Zcal$ has the structure of a $C^{1}$ Banach manifold near $(g,\gamma)\in\Zcal$. This is exactly what happens when $(g,\gamma)$ is in $\Ycal$. For every  $(g,\gamma)\in\Zcal$, we have $\gamma=\gamma'\circ\varphi_d$, where $(g,\gamma')\in\Ycal$, $d\in\Z^+$, and $\varphi_d:S^1\ra S^1$  is the standard degree-$d$ map. $\Zcal$ may fail to be a manifold near $(g,\gamma)$ if $d>1$, as will be discussed in the paper.

\begin{defn}\label{defn:rigidity}
$(g,\gamma)\in\Ycal$ is called $d${\emph{-rigid}} if $\Psi_{g,\gamma\circ\varphi_d}$ is injective, and is called {\emph{super-rigid}} if it is $d$-rigid for every $d\in\Z^+$. A metric $g$ is called {\emph{bumpy}} if every geodesic in $\Ycal(g)$ is super-rigid.
\end{defn}

The subspace $\Gcal^{\bullet }\subset\Gcal$ of bumpy metrics on $M$ is Bair \cite{Abraham}. We reprove this in a slightly stronger form, which is needed in this paper. Given $(g,\gamma)\in\Ycal$, let 
\[\Fcal^d_{g,\gamma}:=A_{2}(N_{\gamma\circ\varphi_d})=T_{\gamma\circ\varphi_d}\Xcal.\] 
Putting these vector spaces together, we obtain a bundle $\Fcal^d\ra \Ycal$. The  operators $\Psi_{g,\gamma\circ\varphi_d}$ give a bundle map $\Psi^d:\Fcal^d\ra \Ecal^d$, where the fibers of the bundle $\Ecal^d\ra \Ycal$ are $\Ecal^d_{g,\gamma}:=A_{0}(N_{\gamma\circ\varphi_d})$.  The rotation map $\psi_d:S^1\ra S^1$ gives an automorphism 
\[\psi_d^*: A_{*}(N_{\gamma\circ\varphi_d})\ra A_{*}(N_{\gamma\circ\varphi_d}),\] 
with $(\psi_d^*)^d=Id $, i.e.  an action of $\Z/d$ on the fibers of  $\Ecal^d$ and $\Fcal^d$. Given a primitive $d$-th root of unity $\lambda$, let $A^\lambda_*(N_\gamma)\subset A_*(N_{\gamma\circ\varphi_d})$ be generated by the sections $\xi$ which admit a {\emph{twin}} $\xi'\in A_*(N_{\gamma\circ\varphi_d}) $ such that 
 \[\psi_d^*(\xi+i\xi')=\lambda\cdot (\xi+i\xi')\in  A_*(N_{\gamma\circ\varphi_d})\otimes_\R\C.\]
It follows that the twin $\xi'$ is also in $A^\lambda_*(N_\gamma)$ and  $A^\lambda_*(N_\gamma) = A^{1/\lambda}_*(N_\gamma)$. For $\lambda=\pm 1$, $A^\lambda_*(N_\gamma) $ consists of the sections $\xi$ so that $\psi_d^*(\xi)=\lambda\cdot\xi$. The vector spaces 
\[\Fcal^\lambda_{g,\gamma}= A^\lambda_2(N_\gamma)\quad\text{and}\quad\Ecal^\lambda_{g,\gamma}= A^\lambda_0(N_\gamma)\]  form the bundles $\Fcal^\lambda\ra \Ycal$ and   $\Ecal^\lambda\ra\Ycal$. We also  obtain the decompositions
 \[\Fcal^d\simeq\bigoplus_{\lambda:\ \lambda^d=1\ \text{and}\ \mathrm{Im}(\lambda)\geq 0}\Fcal^\lambda\quad\quad\text{and}\quad\quad \Ecal^d\simeq\bigoplus_{\lambda:\ \lambda^d=1\ \text{and}\ \mathrm{Im}(\lambda)\geq 0}\Ecal^\lambda.\]
$\Psi^d_{g,\gamma\circ\varphi_d}$ respects these decompositions, and we obtain the homomorphisms $\Psi^\lambda:\Fcal^\lambda\ra \Ecal^\lambda$. Note that the kernel of $\Psi^\lambda$ corresponds to $\lambda$-Jacobi fields  along $(g,\gamma)$.

\begin{thm}\label{thm:super-rigidity}
Given a root $\lambda$ of unity, the moduli space of $\lambda$-Jacobi fields, defined by 
\begin{align*}
\Wcal^\lambda:=\Big\{(g,\gamma,\xi)\ \big|\ (g,\gamma)\in\Ycal\ \ \text{and}\ \ 0\neq \xi\in \Ker(\Psi^\lambda_{g,\gamma})\subset A^\lambda_{2}(N_{\gamma})\Big\}\subset\Fcal^\lambda,
\end{align*}
is a separable $C^{1}$ Banach manifold. The projection map 
\[p^\lambda:\widetilde{\Wcal}^\lambda=\Wcal^\lambda/(S^1\times\R^*)\ra \Gcal\] 
is Fredholm of index $-1$, and for $g$ in the Bair subset $\Gcal^{\lambda}\subset \Gcal$ of the regular values of  $p^\lambda$,
 the pre-image $\widetilde{\Wcal}^{\lambda}(g)$ of $g$ under $p^\lambda$  is empty. 
\end{thm}
\begin{proof}
We assume $\lambda\neq \pm 1$. The proof is similar for $\lambda=\pm1$ (c.f. \cite[Lemma 4.4]{Ef-rigidity}). The differential of $\Psi^\lambda:\Fcal^\lambda\ra \Ecal^\lambda$ at a zero $(g,\gamma,\xi)$ may be projected over the fiber of $\Ecal^\lambda_{g,\gamma}$  to give 
\[D\Psi^\lambda:T_{(g,\gamma)}\Ycal\oplus A^\lambda_{2}(N_{\gamma})\ra A^\lambda_{0}(N_{\gamma}).\]
The surjectivity of $D\Psi^\lambda$ implies the first claim. The restriction of $D\Psi^\lambda$ to  $A^\lambda_{2}(N_{\gamma})$ gives the map
\[D=D_{g,\gamma}^\lambda:A^\lambda_{2}(N_{\gamma})\ra A^\lambda_{0}(N_{\gamma}),\]
which is Fredholm of index $0$. If $D^*$ is the adjoint of $D$, every element in $\Coker(D\Psi^\lambda)$ is represented by some $\zeta\in \Ker(D^*)$ which is orthogonal to the image of $D\Psi^\lambda$. The sections $\xi$ and $\zeta$ admit twins $\xi'$ and $\zeta'$ respectively, so that 
\begin{equation}\label{eq:group-action}
\psi_d^*(\xi+i\xi)=\lambda\cdot(\xi+i\xi)\quad\quad\text{and}\quad\quad \psi_d^*(\zeta+i\zeta)=\lambda\cdot(\zeta+i\zeta).
\end{equation}  
Fix $\theta_0\in S^1$ and choose the local coordinates around $z_0=\gamma(\varphi_d(\theta_0))$ so that $\gamma$ is given by $(\theta,0,...,0)$ and $g(\gamma(\varphi_d(\theta)))$ is given by the identity matrix near $\theta_0$. Fix a function $f$ supported in a neighborhood of $z_0$ so that $f$ and $df$ vanish along $\gamma$. Let $h_f\in\Hcal=T_g\Gcal$ be given by $(h_{jk})_{j,k=1}^m$ in local coordinates, where $h_{jk}=0$, except for $h_{1,1}$ which is equal to $f$. Then 
\[(h_f,0)\in T_{(g,\gamma)}\Ycal\subset \Hcal\oplus A_2(W_\gamma)\]  
and  the $j$-th component of $D\Psi^\lambda(h_f,0,0)$ is $-\sum_k\xi^k\partial_k\partial_jf$, by direct computation. Let $\xi_j^k$, $\zeta^k$ and $(\zeta')^k$ denote the $k$-th coordinates of $(\psi_d^j)^*\xi$, $\zeta$ and $\zeta'$, respectively. $\sum_k\xi^k\partial_k\partial_jf$  is supported over $\cup_{k=0}^{d-1}\psi_d^k(J_\epsilon)$, where $J_\epsilon=(\theta-\epsilon,\theta+\epsilon)$, and 
\begin{align*}
0=\int_{S^1}\Big\langle D\Psi^\lambda(h_f,0,0),(\zeta+i\zeta')\Big\rangle_gd\theta &=-\int_{J_\epsilon} \Big(\sum_{j,k,m=0}^{d-1}\xi_j^k (\psi_d^j)^*(\zeta^m+i(\zeta')^m)\partial_k\partial_mf\Big) d\theta.
\end{align*} 
Moreover, by varying $f$  and using (\ref{eq:group-action}) we conclude that for every $\theta$ with $\zeta(\theta)\neq 0$
\begin{align*}
\sum_{j=0}^{d-1}\xi_j(\theta)\cdot\mathrm{Re}(\lambda^j)=\sum_{j=0}^{d-1}\xi_j(\theta)\cdot\mathrm{Im}(\lambda^j)=0\ \ \ \Rightarrow\ \ \ \sum_{j=0}^{d-1}\lambda^{-j}\cdot (\psi^j_d)^*(\xi)=0.
\end{align*}
Similarly, $\sum_{j=0}^{d-1}\lambda^{-j}\cdot (\psi^j_d) ^*(\xi')=0$ and we thus have 
\[0=\sum_{j=0}^{d-1}\lambda^{-j}\cdot (\psi^j_d) ^*(\xi+i\xi')=\sum_{j=0}^{d-1}\lambda^{-j}\cdot \lambda^j\cdot(\xi+\xi')=d\cdot(\xi+i\xi').\]
This contradiction implies that $\zeta=0$. Therefore,  $D\Psi^\lambda$ is surjective and the zero locus $W^\lambda$ of $\Psi^\lambda$ is transversely cut out as a separable $C^{1}$ Banach  manifold. The index of  
\[p^\lambda:\widetilde{\Wcal}^\lambda=\Wcal^\lambda/(S^1\times\R^*)\ra\Gcal\] 
is $-1$. Let $\Gcal^{\lambda}$ denote the set of regular values of $p^\lambda$, which is a Bair subset of $\Gcal$ by Sard-Smale theorem \cite{Smale}.  The regularity of $g\in \Gcal^{\lambda}$ implies that  $\widetilde{\Wcal}^{\lambda}(g)$  is  empty. 
\end{proof}

The Bair subset 
\[\Gcal^{\bullet }\:=\Gcal^*\cap\big(\bigcap_{\lambda}\Gcal^{\lambda}\big)\subset \Gcal\] 
(where the intersection is over all roots of unity) consists of the metrics $g$ such that all closed $g$-geodesics are super-rigid (i.e. the bumpy metrics). Fix $g_0,g_1\in\Gcal^{\bullet }$. The vector bundles $\Ecal^\lambda$ and $\Fcal^\lambda$ may  be pulled back from $\Ycal$ over $\Ycal_{g_0,g_1}$ via the evaluation map which sends $(t,\vg,\gamma)\in\Ycal_{g_0,g_1}$ to $(\vg(t),\gamma)\in\Ycal$, to give the bundles  $\Ecal^\lambda_{g_0,g_1}$ and $\Fcal^\lambda_{g_0,g_1}$, respectively. Moreover, the operators $\Psi^\lambda_{t,\vg,\gamma}:=\Psi^\lambda_{\vg(t),\gamma}$ give the bundle maps $\Psi^\lambda_{g_0,g_1}:\Fcal^\lambda_{g_0,g_1}\ra \Ecal^\lambda_{g_0,g_1}$. 

\begin{thm}\label{thm:super-rigidity-cobordism}
Fix $g_0,g_1\in\Gcal^{\bullet }$ and the root $\lambda$ of unity. Then 
\begin{align*}
\Wcal^\lambda_{g_0,g_1}:=\left\{(t,\vg,\gamma,\xi)\ \Big|\ (t,\vg,\gamma)\in\Ycal_{g_0,g_1}\ \ \text{and}\ \ 0\neq \xi\in \Ker(\Psi^\lambda_{t,\vg,\gamma})\right\}
\end{align*}
is a separable $C^{1}$ Banach manifold. For $\vg$ in  a Bair subset $\Pcal_{g_0,g_1}^{\lambda}\subset \Pcal_{g_0,g_1}$, 
\begin{displaymath}
\Wcal^\lambda(\vg)=\Wcal_{g_0,g_1}^{\lambda}(M,\vg)=\{(t,\gamma,\xi)\ |\ (t,\vg,\gamma,\xi)\in\Wcal^\lambda_{g_0,g_1}\}
\end{displaymath}
is empty, unless $\lambda=\pm1$. In the latter case,  $\widetilde{\Wcal}^\lambda(\vg)=\Wcal^\lambda(\vg)/(S^1\times\R^*)$ is a $0$-manifold.
\end{thm}
\begin{proof}
The first claim is proved similar to Theorem~\ref{thm:super-rigidity}, by showing the surjectivity of $D\Psi^\lambda_{g_0,g_1}$. $\widetilde{\Wcal}^\lambda(\vg)$ is a $0$-manifold for a regular value $\vg$ of the projection map from  $\Wcal^\lambda_{g_0,g_1}$ to $\Pcal_{g_0,g_1}$. If $\lambda\neq \pm 1$,  $(t,\gamma,\xi)\in \Wcal^\lambda(\vg)$ and its twin $(t,\gamma,\xi')\in \Wcal^\lambda(\vg)$, produce a copy of $\mathbb{RP}^1\simeq S^1$ in $\widetilde{\Wcal}^\lambda(\vg)$. For a regular value $\vg$,  $\Wcal^\lambda(\vg)$ is thus empty unless $\lambda=\pm 1$. 
\end{proof}

For every path $\vg$ connecting $g_0$ to $g_1$ which is in the Bair subset 
\[\Pcal_{g_0,g_1}^{** }:=\bigcap_\lambda\Pcal_{g_0,g_1}^{\lambda}\subset\Pcal_{g_0,g_1},\] 
all closed geodesics in  $\Ycal(\vg)$ are super-rigid, except for a discrete subset $\Ycal_J(\vg)\subset\Ycal(\vg)$. Here, the intersection is over all roots $\lambda$ of unity. We may further refine $\Pcal_{g_0,g_1}^{**}$ and assume that for every $\vg$ in this set, the projection map from $\Ycal_J(\vg)$ to $[0,1]$ is injective, and that for $(t,\gamma)\in\Ycal_J(\vg)$ one of the following happens. Either $\Ker(\Psi_{t,\vg,\gamma\circ\varphi_d})$ is $1$-dimensional for all $d$, and is generated by $\varphi_d^*(\xi)$ for some $\xi\in\Ker(\Psi_{t,\vg,\gamma})$, or it is $1$-dimensional for all even $d$ (generated by $\varphi_{d/2}^*(\xi)$ for some $-1$-Jacobi field $\xi\in\Ker(\Psi_{t,\vg,\gamma\circ\varphi_2})$) and is trivial for odd $d$. In the latter case, since $d\nabla_{t,\vg,\gamma\circ\varphi_2}(\xi)$ is zero in $A_0(N_\gamma)$, the second differentials 
\[d^2\nabla_{t,\vg,\gamma}(\xi\otimes\xi)\quad\text{and}\quad d^2\nabla_{t,\vg,\gamma}(\xi\otimes v_h)\]
 are well-defined in $A_0(N_\gamma)$, and in particular in $C=\Coker(\Psi_{\vg(t),\gamma}^{-1})$. Here, $v_h\in T_{t,\gamma}\Ycal(\vg)$ projects to $h=\partial\vg/\partial s|_{s=t}$ in $\Hcal$ under the differential of the map from $\Ycal(\vg)$ to $\Gcal$ which sends $(s,\gamma')$ to $\vg(s)$. The section $d^2\nabla_{t,\vg,\gamma}(\xi\otimes\xi)$ is of type $+1$, and is thus automatically trivial in $C$.  Again, since $d^2\nabla_{t,\vg,\gamma}(\xi\otimes \xi)$ is zero in $C$, the third differential $d^3\nabla_{t,\vg,\gamma}(\xi\otimes \xi\otimes\xi)$ is well-defined in $C$, and we may refine $\Pcal_{g_0,g_1}^{**}$ one more time to obtain the Bair subset $\Pcal_{g_0,g_1}^{\bullet}$ of $\Pcal_{g_0,g_1}$ such that all the above properties are satisfied, and  for every $(t,\gamma)\in\Ycal_J(\vg)$ which is equipped with a $-1$-Jacobi field $\xi$, both $d^2\nabla_{t,\vg,\gamma}(\xi\otimes v_h)$ and $d^3\nabla_{t,\vg,\gamma}(\xi\otimes\xi\otimes\xi)$ are non-zero in $\Coker(\Psi_{\vg(t),\gamma}^{-1})$. The paths in $\Pcal_{g_0,g_1}^{\bullet }$ are called the {\emph{bumpy paths}} connecting $g_0$ to $g_1$.


\section{Paths of metrics and the weight of closed geodesics}\label{sec:invariance}
Fix $g\in\Gcal$ and $\gamma\in\Zcal(g)$. For simplicity, assume for the moment that $W_\gamma$ may be trivialized along $\gamma$. The general case may be handled with minor modifications. The trivialization of $W_\gamma$ as $W_\gamma=S^1\times \R^n$ may be chosen such that the last component corresponds to $T_\gamma$, and the standard frame on $\R^n$ is orthonormal with respect to $g$.  $\Psi_{g,\gamma}$ may be expressed in these local coordinates by 
\begin{displaymath}
\Psi_{g,\gamma}(\zeta^1,\ldots,\zeta^{n-1})=\sum_{j=1}^{n-1}\Big(\ddot\zeta^j+\sum_{k=1}^{n-1}\big(A^j_k\dot\zeta^k+B^j_k\zeta^k\big)\Big)\cdot e_j.
\end{displaymath}
 Here, $e_j$ is the $j$-th coordinate vector in $\R^n$, while $\zeta^j,\dot\zeta^j$ and $\ddot{\zeta}^j$ denote the $j$-th coordinates of $\zeta,d\zeta/d\theta$ and $d^2\zeta/d\theta^2$, respectively. The pair of matrix-valued functions \[\left(A=\left(A^j_k\right)_{j,k},B=\left(B^j_k\right)_{j,k}\right)\]
  are then the sections of a trivial bundle of rank $2(n-1)^2$ over $S^1$. Denote the contractible space of all $(A,B)$ as above by $\Acal$. Every $(A,B)\in\Acal$ corresponds to an index-$0$ Fredholm map  
\begin{align*}
&L_{A,B}:A_2(\R^{n-1})\ra A_0(\R^{n-1}),\quad\quad L_{A,B}\Big(\sum_j\zeta^je_j\Big):=\sum_{j}\Big(\ddot\zeta^j+\sum_k\big(A^j_k\dot\zeta^k+B^j_k\zeta^k\big)\Big)\cdot e_j.
\end{align*}
 Given $(A,B)\in\Acal$, define the {\emph{nullity}} $\nu(A,B)=\nu(L_{A,B})$ as the dimension of the kernel of $L_{A,B}$ and the {\emph{index}} $\imath(A,B)=\imath(L_{A,B})$ as  the maximum dimension of a finite dimensional subspace $V\subset A_2(\R^{n-1})$ such that the restriction of the quadratic form 
\begin{align*}
q_{A,B}:A_2(\R^{n-1})\otimes A_2(\R^{n-1})\ra \R,\quad\quad
q_{A,B}(\zeta,\xi):=\int_{S^1}\langle L_{A,B}(\zeta),\xi\rangle dt
\end{align*}
to $V$ is positive definite (compare with \cite[Section 1]{White-1}). $(A,B)\in\Acal$ is called {\emph{negatively curved}} if $q_{A,B}$ is negative definite (i.e. $\nu(A,B)=\imath(A,B)=0$). The negatively curved points form a convex subset of $\Acal$. Let $\Acal^r$ denote the subspace of $\Acal$ consisting of $(A,B)$ such that $\nu(A,B)=r$. It is standard to see that  $\Acal^r$ is an (open) analytic sub-variety of $\Acal$ of codimension at least $r$. The closure of $\Acal^1$ (called the {\emph{walls}}) cuts $\Acal^0\subset \Acal$ into chambers. When the walls are crossed transversely, $\imath(A,B)$ changes by $\pm1$. Having fixed the trivialization of $N_\gamma$,  the index $\imath(g,\gamma)=\imath(\Psi_{g,\gamma})$ and the nullity $\nu(g,\gamma)=\nu(\Psi_{g,\gamma})$ for $(g,\gamma)\in\Zcal$ may be defined, and is independent of the choice of the trivialization. A similar construction is possible when $N_\gamma$ is not trivial.

\begin{defn}\label{defn:sign}
For a super-rigid $(g,\gamma)\in\Ycal$ and $d\in\Z^+$ define 
\[\epsilon_d(g,\gamma)=(-1)^{\imath(g,\gamma\circ\varphi_{d})}\in\{+1,-1\}\quad\text{and}\quad
n_d(g,\gamma)=\begin{cases}
\epsilon_1(g,\gamma)&\text{if}\ \ d=1\\ (\epsilon_2(g,\gamma)-\epsilon_1(g,\gamma))/2&\text{if}\ \ d=2\\  0&\text{if}\ \  d>2
\end{cases}.
\] 
The integer $n(g,\gamma\circ\varphi_d)=n_d(g,\gamma)$ is called the {\emph{weight}} of the closed $g$-geodesic $\gamma\circ \varphi_d$. 
If  $\Gamma\subset\Zcal(g)$ is  a finite set of multiple covers of super-rigid $g$-geodesics, define $n(g,\Gamma)=\sum_{\gamma\in\Gamma}n(g,\gamma)$. 
\end{defn}

If $g$ is negatively curved in a neighborhood of $\gamma\in\Ycal(g)$, $n_d(g,\gamma)$ is $1$ if $d=1$ and is $0$ otherwise. 

\begin{thm}\label{thm:invariance-count-function}
Fix  $g_0,g_1\in\Gcal^{\bullet }$ and  $\vg\in\Pcal^{\bullet }_{g_0,g_1}$. If $\vGamma\subset \Zcal(\vg)$ is  compact and open, and 
\[\Gamma_j=\vGamma\cap\Zcal(g_j)\quad\quad \text{for}\ \  j=0,1,\] 
then $n(g_1,\Gamma_1)=n(g_0,\Gamma_0)$.
\end{thm}
\begin{proof}
 Let $q$ denote the projection from $\widetilde{\Ycal}(\vg)$ over  $[0,1]$ and  $g_t=\vg(t)$.  $\vGamma=\cup_d \vGamma^d$, where $\vGamma^d\subset\vGamma$ consists of the degree-$d$ covers of somewhere injective geodesics. Let $\tGamma^d$ denote the set of all  $(t,\gamma)$ such that $(t,\gamma\circ\varphi_d)\in\vGamma^d$. Then $\tGamma=\cup_d\tGamma^d$ is an open subset of $\Ycal(\vg)$. For $t\in[0,1]$, denote the set of all pairs $(t,\gamma)$ in $\vGamma,\vGamma^d$ and $\tGamma^d$ by $\Gamma_t,\Gamma_t^d$ and $\tGamma_t^d$, respectively. Since $\vGamma$ is compact,  $(\Ycal_J(\vg)\cap\tGamma)/S^1$ is finite and its image under $q$ is a finite set of values $0<t_1<\cdots<t_m<1$. If $(t,\gamma)\in\tGamma$ for some 
 \[t\in[0,1]\setminus\{t_1,t_2,\ldots,t_m\},\]  
 $\Psi_{g_t,\gamma\circ\varphi_d}$ is bijective. Therefore, there is a map 
 \[\tau:I=(t-\epsilon,t+\epsilon)\ra \widetilde{\Ycal}(\vg)\] 
 represented by $s\mapsto(s,\gamma_s)\in\Ycal(\vg)$, such that $\gamma_t=\gamma$, $q\circ\tau$ is the identity map over $I$,  and the only points in a neighborhood of $(t,\gamma\circ\varphi_d)$ in $\widetilde{\Zcal}(\vg)$ are $\{(s,[\gamma_s\circ \varphi_d])\}_{s\in I}$.  Moreover, $\epsilon_k(g_s,\gamma_s)=\epsilon_k(g_t,\gamma_t)$ for $s\in I$ and $k\in\Z^+$, which implies that $n(g_s,\Gamma_s)$ is locally constant on $[0,1]\setminus\{t_1,\ldots,t_m\}$.\\

Next, let $t=t_k$ and $(t,[\gamma])\in\widetilde{\Ycal}_J(\vg)$ be the only curve in $q^{-1}\{t\}\cap\widetilde{\Ycal}_J(\vg)$. Then there is either a  $-1$-Jacobi field $\xi$ along $\gamma$ (case $1$), or a  $+1$-Jacobi field $\xi$ (case $2$). In case $1$, $\Psi_{g_t,\gamma}$ is bijective. Therefore, $(t,[\gamma])$ is a regular point of the projection map $q:\widetilde{\Ycal}(\vg)\ra [0,1]$ and for a sufficiently small $\epsilon>0$, there is a local inverse 
\[\tau_1:I=(t-\epsilon,t+\epsilon)\ra \widetilde{\Ycal}(\vg),\]
 represented by $\tau_1(s)=(s,[\gamma_s])$, where $(s,\gamma_s)\in\Ycal(\vg)$ and $(t,[\gamma])$ is not the limit of any sequence in $\widetilde{\Zcal}(\vg)\setminus \tau_1(I)$. Fix an identification of a tubular neighborhood of $\gamma$ in $M$ with a neighborhood of the zero section in $N_\gamma$. If $(s,\beta)\in \Zcal(\vg)$ and the image of $\beta$ is sufficiently close to the image of $\gamma\circ\varphi_2$, $\beta\circ\rho$ is identified with a section in $A_2(N_{\gamma\circ\varphi_2})$ and $N_{\beta\circ\rho}$ is naturally identified with $N_{\gamma\circ \varphi_2}$. Moreover, $\rho$ is determined by $\beta$ upto a composition with the rotation $\psi_2:S^1\ra S^1$ (by $\pi$).\\

 Sending $(s,\beta)$ to $\nabla(g_s,\beta)\in A_0(N_\beta)$ gives a function 
 \[G:I\times A_2(N_{\gamma\circ\varphi_2})\ra A_0(N_{\gamma\circ\varphi_2}),\]
  and every closed geodesic in a  neighborhood of $(t,[\gamma\circ\varphi_2])$ in $\widetilde{\Zcal}(\vg)$ corresponds to a zero of $G$ in  a neighborhood of $(t,0)$.  As discussed above, $(s,\beta)$ and $(s,\beta\circ\psi_2)$ correspond to the same point in $\widetilde{\Zcal}(\vg)$, while they correspond to different zeros of $G$ if $\beta$ is not the double-cover of another geodesic. This observation implies that there is a map $\imath'$ and an involution $\jmap'$ with  
  \[\imath':I\times A_2(N_{\gamma\circ\varphi_2})\ra A_2(N_{\gamma\circ\varphi_2})\quad\text{and}\quad \jmath':A_0(N_{\gamma\circ\varphi_2})\ra A_0(N_{\gamma\circ\varphi_2}),\]
  which satisfy $\imath'(s,\imath'(s,\zeta))=\zeta$ and $\jmath'(0)=0$, such that $G(s,\imath'(s,\zeta))=\jmath'(F(s,\zeta))$ for all $(s,\zeta)\in I\times A_2(N_{\gamma\circ\varphi_2})$. The differential of $G$ at $(t,0)$ is the restriction of $\wPsi_{g_t,\gamma\circ\varphi_2}$ to $\langle h\rangle_\R\oplus A_2(N_{\gamma\circ\varphi_s})$. Therefore, 
  \[\Ker(dG_{t,0})=\R\oplus\Ker(\Psi^{-1}_{g_t,\gamma})\quad\text{and}\quad\Coker(dG_{t,0})=\Coker(\Psi^{-1}_{g_t,\gamma}),\]
 where the generators of the kernel are given as $v_h=(h,\varphi_2^*(\xi'))$ and $(0,\xi)$, with $\xi'\in A_2^{+1}(N_\gamma)$ and $\xi\in A_2^{-1}(N_\gamma)$. Let $H$ denote $A_2(N_{\gamma\circ\varphi_2})/\langle\xi\rangle$ and $H'$ denote the image of $\Psi_{g_t,\gamma\circ\varphi_2}$ in $A_0(N_\gamma)$. We thus obtain the isomorphisms
\[A_2(N_{\gamma\circ \varphi_2})\simeq  \R\oplus H=\langle \xi\rangle_\R\oplus H \quad\quad\text{and}\quad\quad A_0(N_{\gamma\circ \varphi_2})\simeq \R\oplus H'=\Coker(\Psi^{-1}_{g_t,\gamma})\oplus H'.\]
Let $G(s,r,y)=(G_1(s,r,y),G_2(s,r,y))$, where $y\in H$ and $G_2(s,r,y)\in H'$. The differential of  $G_2$ with respect to $y$ at $(t,0,0)$ is an invertible map from $H$ to $H'$. Therefore,  $G_2(s,r,y)=0$ may be solved to find $y=y(s,r)$. Then  $G(s,r,y)=0$ is equivalent to  $F(s,r):=G_1(s,r,y(s,r))=0$. We already know that $(s,\gamma_s\circ\varphi_2)$ (i.e. the double cover of the path $\tau_1:I\ra \Ycal(\vg)$) corresponds to a path $s\mapsto (s,r(s))$ for $s\in I$ such that $F(s,r(s))=0$. After a change of variable in the domain of $F$ (and for simplicity), we may assume that $r(s)=0$ for all $s\in I$.\\

 It  follows that $\imath'$ induces a map $\imath:I\times \R\ra \R$ satisfying $\imath(s,\imath(s,r))=r$ and $\jmath'$ induces an involution $\jmath:\R\ra\R$ with $\jmath(0)=0$, such that $F(s,\imath(s,r))=\jmath(F(s,r))$ for all $(s,r)$ close to $(t,0)$. Moreover, $\imath(s,0)=0$ for all $s$. A neighborhood of $(t,[\gamma\circ\varphi_2])$ in $\widetilde{\Zcal}(\vg)$ is then identified with a neighborhood of $(t,0)$ in the zero set of $F$, which is mod out by the action of the aforementioned involution. Since $\xi\in\Ker(\Psi^{-1}_{g_t,\gamma})$, it follows that $\partial F/\partial r(t,0)=0$. Moreover,  $\partial^2F/\partial r^2(t,0)$ is automatically zero, since it is the image of a section of type $+1$ in $\Coker(\Psi_{g_t,\gamma}^{-1})$. Furthermore, $\partial^3F/\partial r^3(t,0)\neq 0$ and $\partial^2F/\partial r\partial s(t,0)\neq 0$ since $\vg\in\Pcal^{\bullet}_{g_0,g_1}$. Thus, in a neighborhood of $(t,0)$,  there is a presentation 
\begin{align*}
&F(s,r)=r((s-t)f(s,r)-r^2g(r)),\quad\quad\text{where}\quad f:I\times\R\ra \R,\ \  g:\R\ra \R\ \ \text{and}\ \ g(0),f(t,0)\neq 0.
\end{align*}

Therefore, there is an interval $J$, which is either $(t-\epsilon,t)$ or $(t,t+\epsilon)$, and  the maps 
\[\tau_1:I\ra \widetilde{\Ycal}(\vg)\quad\text{and}\quad \tau_2:J\ra \widetilde{\Ycal}(\vg)\] 
giving a local model for $\widetilde{\Zcal}(\vg)$ near $(t,[\gamma\circ\varphi_2])$ via (c.f. \cite[Lemma 5.11]{Taubes}):
\begin{itemize}
\item $q\circ \tau_1$ and $q\circ \tau_2$ are the identity maps of $I$ and $J$, respectively.
\item $\tau_1(t)=(t,[\gamma])$ and $(t,[\gamma\circ\varphi_2])$ is the limit of $\tau_2(s)$ as $s$ approaches $t$. Moreover, $(t,[\gamma])$ and $(t,[\gamma\circ\varphi_2])$ are not the limit points of  $\widetilde{\Ycal}(\vg)\setminus \tau_1(I)$ and   $\widetilde{\Ycal}(\vg)\setminus \tau_2(J)$, respectively.
\item The kernels of $\Psi^\lambda_{\tau_1(s)}$ and $\Psi^\lambda_{\tau_2(s)}$ are trivial for all $\lambda$ and $s\neq t$. The kernel of $\Psi^\lambda_{\tau_1(t)}$  is trivial for $\lambda\neq -1$, while it is $1$-dimensional for $\lambda=-1$.
\end{itemize}  
For simplicity, let us assume that $J=(t-\epsilon,t)$. The other case is  similar. Let $\tau_1(s)=(s,[\gamma_s])$ and $\tau_2(s)=(s,[\gamma'_s])$. Then $\epsilon_1(g_s,\gamma'_s)=-\epsilon_2(g_s,\gamma_s)$, as the kernels of the operators in the path
\[\left\{\Psi_{g_s,\gamma_s\circ \varphi_2}\right\}_{s\in(t-\epsilon,t]}\cup\left\{\Psi_{g_{2t-s},\gamma'_{2t-s}}\right\}_{s\in(t,t+\epsilon)}\]
are trivial except at $s=t$, where it is $1$-dimensional. Similarly, it follows that 
\[\epsilon_2(g_s,\gamma_s')=-\epsilon_4(g_s,\gamma_s).\]
Moreover, for every $g$-geodesic $\beta$,  $\epsilon(\Psi_{g,\beta\circ \varphi_4})=\epsilon(\Psi_{g,\beta\circ \varphi_2})$, as  a path $(g^s,\beta^s)$ from $(g,\beta)$ to the negatively curved chamber may be chosen so that $\Psi^\lambda_{g^s,\beta^s}$ has trivial kernel for $\lambda=\pm i$.  Therefore, $\epsilon_2(g_s,\gamma_s)=\epsilon_4(g_s,\gamma_s)=-\epsilon_2(g_s,\gamma_s')$.  For $s\in (t,t+\epsilon)$ we have 
\[\epsilon_1(g_s,\gamma_s)=\epsilon_1(g_{2t-s},\gamma_{2t-s})\quad\quad\text{and}\quad\quad\epsilon_2(g_s,\gamma_s)=-\epsilon_2(g_{2t-s},\gamma_{2t-s}).\] 

\begin{figure}
\def\svgwidth{\textwidth}
\begin{center}
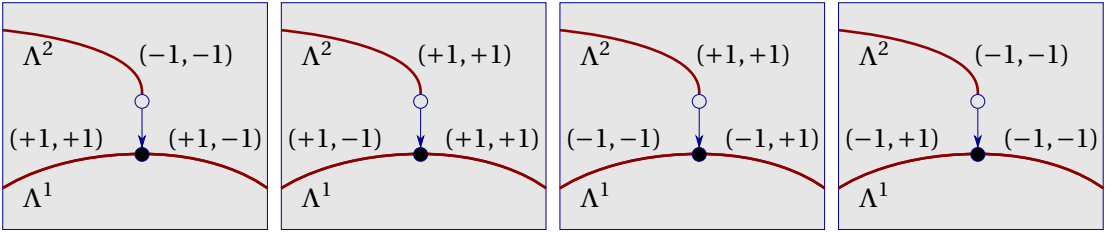
 \caption{\label{fig:local-contribution}
{A sequence in $\tGamma^{2}/S^1$ may converge to the double cover of $[\gamma]\in\tGamma^1/S^1$. The pair $(\epsilon_1,\epsilon_2)$  in a neighborhood $\tau_1(t-\epsilon,t+\epsilon)$ of $(t,[\gamma])$ in $\tGamma^1/S^1$ and a neighborhood  $\tau_2(t-\epsilon,t)\subset \tGamma^{2}/S^1$ of $(t,[\gamma\circ\varphi_2])$ follows one of the $4$ illustrated patterns.}}
\end{center}
\end{figure}

Possible values for $(\epsilon_1,\epsilon_2)$ over this local model are illustrated in Figure~\ref{fig:local-contribution} for $J=(t-\epsilon,t)$. If $d$ is odd, since $\Ker(\Psi_{g_t,\gamma\circ\varphi_d})$ is trivial,  a neighborhood of $(t,\gamma\circ\varphi_d)$ in $\Zcal(\vg)$ consists only of the pairs $(s,\gamma_s\circ\varphi_d)$ for $s\in I$. Similarly, when $d$ is even, since  $\Ker(\Psi_{g_t,\gamma\circ\varphi_d})$ is identified with $\varphi_{d/2}^*(\Ker(\Psi_{g_t,\gamma\circ\varphi_2}))$, a neighborhood of $(t,\gamma\circ\varphi_d)$ in $\Zcal(\vg)$ consists only of $(s,\gamma_s\circ\varphi_d)$ for $s\in I$ and the pairs $(s,\gamma'_s\circ\varphi_{d/2})$ for $s\in J$. Thus, the contribution of the aforementioned local model to $n(g_s,\Gamma_s)$ remains constant for $s\in I\setminus\{t\}$. This is trivial if $k>2$ from definition, and may be checked from the above model in  cases $k=1$ and $k=2$, separately, completing the study of case $1$.\\

In case $2$, $\Ker(\Psi_{g_t,\gamma\circ\varphi_d})$ and $\Coker(\Psi_{g_t,\gamma\circ\varphi_d})$ may be identified with the pull-backs of $\Ker(\Psi_{g_t,\gamma})$ and $\Coker(\Psi_{g_t,\gamma})$ under $\varphi_d^*$, and are both $1$-dimensional. Moreover, the image of $h=\partial \vg/\partial s|_{s=t}\in\Hcal$ under $\Psi^m_{g,\gamma}$ is non-trivial in $\Coker(\Psi_{g_t,\gamma})$. Note that $\Psi^m_{g_t,\gamma\circ\varphi_d}=\varphi_d^*\circ \Psi^m_{g_t,\gamma}$, which implies that the image of $h$ under $\Psi^m_{g,\gamma\circ\varphi_d}$ is non-trivial in $\Coker(\Psi_{g_t,\gamma\circ\varphi_d})$. In particular, $\widetilde{\Zcal}(\vg)$ has the structure of a $1$-manifold near $(t,[\gamma\circ\varphi_d])$ and is obtained by composing the geodesics in a neighborhood of $(t,[\gamma])\in\widetilde{\Zcal}(\vg)$ with the degree-$d$ map $\varphi_d$. An argument similar to case $1$ implies that $t$ is a critical value for $q$, and there is a an interval $J$, equal to $(t-\epsilon,t)$ or $(t,t+\epsilon)$, and the maps $\tau^{\pm}:J\ra \widetilde{\Ycal}(\vg)$ such that the following are satisfied (c.f. \cite[Lemma 5.9]{Taubes}):
\begin{itemize}
\item $q\circ\tau^+$ and $q\circ\tau^-$ are both the identity map of $J$.
\item $\tau^\pm(s)=(s,[\gamma^\pm_s])$ and $\gamma^\pm_s$ converge to the same curve $(t,\gamma_t=\gamma)\in{\Ycal}(\vg)$. Moreover, $(t,[\gamma\circ\varphi_d])$ is not the limit of any sequence in $\widetilde{\Zcal}(\vg)\setminus \{(s,[\gamma^\pm_s\circ\varphi_d]\}_{s\in J}$.
\item The kernel of $\Psi^\lambda_{\tau^\pm(s)}$ is trivial for $s\in J$ and any root $\lambda$ of unity.
\end{itemize}
Therefore, $\epsilon_j(g_s,\gamma_s^+)=-\epsilon_j(g_s\gamma_s^-)$ for $j\in\Z^+$. In particular, $n(g_s,\Gamma_s)$ remains constant for $s\neq t$ in a neighborhood $(t-\epsilon,t+\epsilon)$ of $t$. The above study, implies that $n(g_s,\Gamma_s)$ remains constant as  $s\in[0,1]$ passes any of the values $0<t_1<\cdots<t_m<1$. As we have already shown that $n(g_s,\Gamma_s)$ is constant on each connected component of $[0,1]\setminus\{t_1,\ldots,t_m\}$, the proof is complete. 
\end{proof}
\section{Non-generic Riemannian metrics and counting geodesics}\label{sec:generalizations}
Choose $g\in\Gcal$ and let $\Gamma$ be a compact and open subset of $\Zcal(g)$. There are thus bounded open subsets $\Ucal=\Ucal_{g,\Gamma}$ and $\Ucal'=\Ucal'_{g,\Gamma}$ of $\Xcal$ with $\overline{\Ucal}\subset\Ucal'$ and 
\[\Gamma=\Zcal(g)\cap\Ucal=\Zcal(g)\cap\Ucal'.\] If there is a sequence $\{g_j\}_j$ in $\Gcal$ which converges to $g$ such that $\Gamma_j=\Zcal(g_j)\cap \Ucal$ is not closed, then there are closed $g_j$-geodesics $\{\gamma^k_j\}_k\in\Ucal$ which converge to the $g_j$-geodesic $\gamma_j\in\Ucal\cap\Zcal(g_j)$. By passing to a sub-sequence, we may assume that $\{\gamma_j\}_j$ converges to a closed $g$-geodesic $\gamma\in\overline{\Ucal}\cap\Zcal(g)=\Gamma$. If $k_j$ is sufficiently large, the sequence $\{\gamma_j^{k_j}\}_j$ (which is outside $\Ucal$) will also converge to $\gamma$. But $\Xcal\setminus\Ucal$ is closed, and thus $\gamma$ should be in $\Xcal\setminus\Ucal$. Therefore, $\gamma\in\Gamma\cap(\Xcal\setminus\Ucal)=\emptyset$.  This contradiction implies that there is an open and path connected neighborhood $U=U_{g,\Gamma}$ of $g$ in $\Gcal$ such that for every $g'\in U$, the set $\Gamma_{g'}=\Zcal(g')\cap\Ucal$ is closed, and is thus compact and open (since $\Ucal$ is open and bounded). The choice of the open sets $\Ucal_{g,\Gamma}$ and $U_{g,\Gamma}$ is of course not unique. 

\begin{defn}\label{defn:contribution}
If $g\in\Gcal$ and $\Gamma$ is a compact and open subset of $\Zcal(g)$, the {\emph{weight}} $n(g,\Gamma)$ of $\Gamma$  is defined equal to $n(g',\Gamma')$, where $g'$ is an arbitrary metric in $\Gcal^{\bullet }\cap U_{g,\Gamma}$ and $\Gamma'=\Zcal(g')\cap\Ucal_{g,\Gamma}$. 
\end{defn}

\begin{thm}\label{thm:invariance}
If $\Gamma$ is a compact and open subset of $\Zcal(g)$ for some $g\in\Gcal$, then $n(g,\Gamma)$ is well-defined and independent of the choices made in Definition~\ref{defn:contribution}. Moreover, if  $\vGamma$ is a compact and open  subset of $\Zcal(\vg)$ for some $\vg\in\Gcal_{g_0,g_1}$, then $n(g_0,\Gamma_0)=n(g_1,\Gamma_1)$, where $\Gamma_j=\vGamma\cap\Zcal(g_j)$ for $j=0,1$.
\end{thm}

\begin{proof}
If $g_0,g_1\in \Gcal^{\bullet }\cap U_{g,\Gamma}$ are a pair of bumpy metrics in $U=U_{g,\Gamma}$, choose a bumpy path $\vg$ in $U$ which connects them. Then 
\[\vGamma=\Zcal(\vg)\cap ([0,1]\times \Ucal)\] 
is open and compact, where $\Ucal=\Ucal_{g,\Gamma}$. If $\Gamma_j=\Ucal\cap\Zcal(g_j)$ for $j=0,1$, Theorem~\ref{thm:invariance-count-function} implies that $n(g_0,\Gamma_0)=n(g_1,\Gamma_1)$. Therefore,  $n(g,\Gamma)$ is independent of the choice of $g'$ in $U_g$. Moreover, the independence from the choice of $U_{g,\Gamma}$ and $\Ucal_{g,\Gamma}$ follows as well, since we may always pass to the intersection of two different selected open subsets.\\

 The above argument also shows that for the second claim, we may further assume that $g_0,g_1\in\Gcal^\bullet$. Since $\vGamma$ is compact and open, there are bounded open sets $\Ucal$ and $\Ucal'$ in $[0,1]\times \Xcal$ so that $\overline{\Ucal}\subset\Ucal'$ and $\vGamma=\Zcal(\vg)\cap\Ucal=\Zcal(\vg)\cap\Ucal'$. As discussed after Definition~\ref{defn:contribution},  there is a neighborhood $U$ of $\vg$ in $\Pcal_{g_0,g_1}$ such that for every other path $\vg'\in U$, $\vGamma'=\Zcal(\vg')\cap\Ucal$ is compact and open. If $\vg'\in\Pcal^\bullet_{g_0,g_1}$, Theorem~\ref{thm:invariance-count-function} implies that $n(g_0,\Gamma_0)=n(g_1,\Gamma_1)$, completing the proof of the theorem.
\end{proof}

\begin{example} Given the real numbers $0<a_0\leq\cdots\leq a_n$,  the $n$-dimensional ellipsoid is defined by
\[M=E_{a_0,\ldots,a_n}=\big\{(x_0,\ldots,x_n)\in\R^{n+1}\ \big|\ \left({a_0x_0}\right)^2+\cdots+\left({a_n x_n}\right)^2=1\big\},\]
and is equipped with the restriction $g=g_{a_0,\ldots,a_n}$ of the Euclidean metric. Choose $a_0,\ldots,a_n$  sufficiently close to $1$. Given $0\leq j<k\leq n$, let $\gamma_{j,k}$ be the geodesic obtained as the intersection  of $M$ with the plane $P_{j,k}$, which is defined as the set of all $(x_0,\ldots,x_n)\in\R^{n+1}$ with $x_j=x_k=0$. If $a_0,\ldots,a_n$ are different, $\gamma_{j,k}$ (parametrized by arc length) are the only simple closed geodesics on $M$ (upto change of orientation). By \cite[Theorem 4.5]{White-1}, $\imath(g,\gamma_{j,k})=n+j+k-2$ and $n_1(g,\gamma_{j,k})=(-1)^{n+k+j}$.\\

When $a_0=\cdots=a_n=1$, we recover the standard metric $g_{S^n}$ on the sphere $S^n\subset \R^{n+1}$. Let $\Gamma_d$ denote the set of $d$-covers of all great circles on $S^n$, which are all the closed geodesics of length $2d\pi$. By perturbing the metric to one of the ellipsoid metrics discussed above we obtain
\[n(g_{S^n},\Gamma_1)
=2\sum_{0\leq j<k\leq n}(-1)^{n+k+j}=2(-1)^{n+1}\left\lfloor\frac{n+1}{2}\right\rfloor=(-1)^{n+1}\chi(\Gamma_1/S^1).\] 
The last equality follows since  $\Gamma_1/S^1$ is the double-cover of the Grassmannian $G(2,n+1)$ (note that the geodesics in $\Gamma_1/S^1$ have a well-defined orientation). The projection map from $\widetilde{\Ycal}$ to $\Gcal$ is proper in a neighborhood of $g_{S^n}$ and its degree is $(-1)^{n+1}\chi(\Gamma_1/S^1)$, by \cite[Theorem 5.1]{White-1}.  The weights $n(g_{S^n},\Gamma_d)$ for $d>1$ may be computed similarly, and are all equal to zero. Therefore, we have
\[\pi_{g_{S^n}}:\R\setminus 2\pi\Z\ra \Z,\quad\quad\pi_{g_{S^n}}(L)=\begin{cases}0& L<2\pi\\ 2(-1)^{n+1}\lfloor(n+1)/2\rfloor& L>2\pi \end{cases}.\]
\end{example}

Suppose that $M$ is a surface of genus $h$ and let $g\in\Gcal$. When $h=0$,  the number $N_g(L)$ of closed $g$-geodesics of length at most $L$ satisfies the following inequality (\cite{KliTak,Moser, Hingston-1} and \cite{Hingston}):
\[\liminf_{L\ra \infty} \ \frac{N_g(L)}{(L/\log(L))}>0.\]
The count function $\pi_g=\pi_M(g,\cdot)$ is different from $N_g$ in this case, unlike the case of $h>1$ and negatively curved $g$, where $\pi_g=N_g$. The {\emph{prime geodesic theorem}} states that $\pi_g$ is asymptotic to $e^{\hslash(g) L}/(\hslash(g) L)$, where $\hslash(g)$ is the topological entropy of  $g$ (\cite{PS} and  \cite{Margulis}). For arbitrary $g\in\Gcal$, it makes sense to ask about the asymptotic behavior of $\pi_g$. The aforementioned results suggest comparing the asymptotics of $\pi_g$  with $e^{\hslash(g) L}/(\hslash(g) L)$.

\end{document}